\documentclass[10pt]{amsart}

\usepackage{amsmath, amscd, amssymb, euler}
\usepackage[frame,cmtip,arrow,matrix,line,graph,curve]{xy}
\usepackage{graphpap, color,bbm, verbatim}
\usepackage[mathscr]{eucal}
\usepackage[normalem]{ulem}
\usepackage{pgf,tikz}
\usetikzlibrary{arrows}

\numberwithin{equation}{section}

\newcommand{\CC}{\mathbb{C}}

\newcommand{\PP}{\mathbb{P}}
\newcommand{\QQ}{\mathbb{Q}}

\newcommand{\ZZ}{\mathbb{Z}}

\newcommand{\bB}{\mathbf{B}}

\newcommand{\bM}{\mathbf{M}}

\newcommand{\cal}{\mathcal}

\def\cF{{\cal F}}

\def\cO{{\cal O}}

\newcommand{\ses}[3]{0\lr{#1}\lr{#2}\lr{#3}\lr 0}

\def\lr{\rightarrow}

\input xy
\xyoption{all}

\DeclareMathOperator{\Ext}{Ext} 
\DeclareMathOperator{\Hom}{Hom} 

\newtheorem{prop}{Proposition}[section]
\newtheorem{theo}[prop]{Theorem}
\newtheorem{lemm}[prop]{Lemma}
\newtheorem{coro}[prop]{Corollary}

\theoremstyle{definition}

\newtheorem{rema}[prop]{Remark}

\title[Virtual Poincar\'e polynomial of the space of stable pairs]{Virtual Poincar\'e polynomial of the space of stable pairs supported on quintic curves}

\author{Kiryong Chung}
\address{School of Mathematics, Korea Institute for Advanced Study, Seoul 130-722, Korea}
\email{krjung@kias.re.kr}

\begin{document}
\begin{abstract}
Let $\mathbf{M}^{\alpha}(d,\chi)$ be the moduli space of $\alpha$-stable pairs $(s,F)$ on the projective plane $\mathbb{P}^2$ with Hilbert polynomial $\chi(F(m))=dm+\chi$.
For sufficiently large $\alpha$ (denoted by $\infty$), it is well known that the moduli space is isomorphic to the relative Hilbert scheme of points over the universal degree $d$ plane curve. For the general $(d,\chi)$, the relative Hilbert scheme does not have a bundle structure over the Hilbert scheme of points. In this paper, as the first non trivial such a case,  we study the wall crossing of the $\alpha$-stable pairs space when $(d,\chi)=(5,2)$. As a direct corollary, by combining with Bridgeland wall crossing of the moduli space of stable sheaves, we compute the virtual Poincar\'e polynomial of $\mathbf{M}^{\infty}(5,2)$.
\end{abstract}

\maketitle
\section{Introduction}
\subsection{Introduction and results}
By definition, a pair $(s, F)$ consists of a sheaf $F$ on $\PP^2$ and one-dimensional subspace $s \subset H^0(F)$. Let us fix $\alpha \in \QQ[m]$ with a positive leading coefficient. A pair $(s,F)$ is called \emph{$\alpha$-semistable} if $F$ is pure and for any subsheaves $F'\subset  F$, the inequality
$$
\frac{\chi(F'(m))+\delta\cdot\alpha}{r(F')} \leq \frac{\chi(F(m))+\alpha}{r(F)}
$$
holds for $m\gg 0$. Here $r(F)$ is the leading coefficient of the Hilbert polynomial $\chi(F(m))$ and $\delta=1$ if the section $s$ factors through $F'$ and $\delta=0$ otherwise.
When the strict inequality holds, we say $(s,F)$ is $\alpha$-stable.

With the help of the general result of the geometric invariant theory (\cite{Mumford2}), Le Potier (\cite[Theorem 4.12]{lepot2}) proved that there exist projective schemes $\bM^{\alpha}(d,\chi)$ parameterizing $\alpha$-stable pairs $(s,F)$ such that $F$ has Hilbert polynomial $P(m)=dm+\chi$. Also, M. He (\cite{mhe}) studied the wall crossings (or flips) of the moduli spaces $\bM^{\alpha}(d,\chi)$ as $\alpha$ varies. In two extremal case,
\begin{itemize}
\item If $\mbox{deg}(\alpha)\geq 2$, then $\bM^{\alpha:=\infty}(d,\chi)$ is isomorphic to the relative Hilbert scheme of $n=\chi-\frac{d(3-d)}{2}$ points on the universal degree $d$ curve (\cite[\S 4.4]{mhe}, \cite[Proposition B.8]{PT}). Let us denote by $\bB(d,n)$ the relative Hilbert scheme. When $\alpha=\infty$, $\alpha$-stable pairs are precisely stable pairs in the sense of Pandharipande-Thomas (\cite{PT}).
\item If $\alpha$ is sufficiently small (denoted by $\alpha=+$), the moduli space has a natural forgetful morphism
$$
\xi: \bM^{+}(d,\chi) \longrightarrow \bM(d, \chi)
$$
which associates to the $0^+$-stable pair $(s,F)$ the sheaf $F$. The later moduli space $\bM(d, \chi)$ parameterizes $S$-equivalent classes\footnote{Two semistable sheaves are $S$-equivalent if they have isomorphic Jordan-H\"{o}lder filtration.} of semistable sheaves with Hilbert polynomial $dm+\chi$ (\cite{HL}).
\end{itemize}
When $\chi=1$, the moduli space $\bM(d, 1)$ is a smooth projective variety of dimension $d^2+1$. Relating with the curve counting invariants on the (open) Calabi-Yau threefolds, S. Katz (\cite{katz_gv}) conjectured the signed topological Euler number
$$
(-1)^{d^2+1}\cdot e(\bM(d,1))=n_{0,d}
$$
is exactly the genus $0$, BPS (or GV)-number $n_{0,d}$ on the local $\PP^2$ (i.e., the total space of the canonical line bundle of $\PP^2$). For the introduction of these subjects, see \cite{PT14}. For the local $\PP^2$, several authors confirmed that this conjecture holds (\cite{sahin, yuan2, cc1, cm, cc2}) for the lower degree cases through several different methods.
Specially, in \cite{cc1}, when $d\leq5$ and $\chi=1$, the authors show that the moduli spaces $\bM^{\alpha}(d,1)$ are birational among each other and thus we obtain the cohomology group of the space $\bM(d,1)$ by studying the wall crossing of the moduli spaces $\bM^{\alpha}(d,1)$. In this case, the work is well-going since the relative Hilbert scheme $\bB(d,n)$ has a projective bundle structure and all of the wall crossings are \emph{simple}, that is, the length of the JH-filtrations of $\alpha$-stable pairs is two. But, for the large $(d,\chi)$, the wall crossings among the $\alpha$-stable pairs space become very complicate because the wall crossing may not be simple. Also it is hard to understand the geometry of the relative Hilbert scheme $\bB(d,n)$ (cf. \cite{cc1}). Hence we need more careful study to get some geometric information of the space $\bM(d,\chi)$ from the $\alpha$-stable pairs spaces or its wall crossings. In this paper, we study the wall crossings when $(d,\chi)=(5,2)$, which is the first case such that the relative Hilbert scheme is not a projective bundle and the wall crossings may not be simple. That is, we will show that
\begin{theo}\label{mainthm}
\begin{enumerate}
\item There are five wall crossings between $\bM^{\infty}(5,2)$ and $\bM^{+}(5,2)$; the walls occur at $\alpha=18$, $13$, $8$, $3$ and $\frac{1}{2}$.
\item The forgetful map $\bM^+(5,2)\lr \bM(5,2)$ is a projective bundle on $\bM^{+}(5,2)_2$ with fiber $\PP^1$. In the complement of $\bM^{+}(5,2)_2$, it is a $\PP^2$-bundle map.
\end{enumerate}
Here $\bM^{+}(5,2)_2$ is the locus of $0^+$-stable pairs $(s,F)$ with $h^0(F)=2$.
\end{theo}
On the other hand, the moduli space $\bM(5,2)$ has another wall crossings, that is, the Bridgeland wall crossing. This was done in a general setting by many authors (for example, \cite{woolf,bertram}). In order to get the cohomology group of the space $\bM(5,2)$ from the Bridgeland wall crossing as it was done in \cite{cc2}, it is essential to know the final birational (or wall crossing model) of the moduli space $\bM(5,2)$. By studying the nef cone of one of the birational model of $\bM(5,2)$, we obtain
\begin{prop}
The final birational model of $\bM(5,2)$ is isomorphic to the Grassmannian variety $Gr(2,15)$.
\end{prop}
The wall crossings of two different types are summarized into the following diagram. Let $\bM^{\alpha}(5,2):=\bM_5^\alpha$ and $\bM(5,2):=\bM_5$.
$$
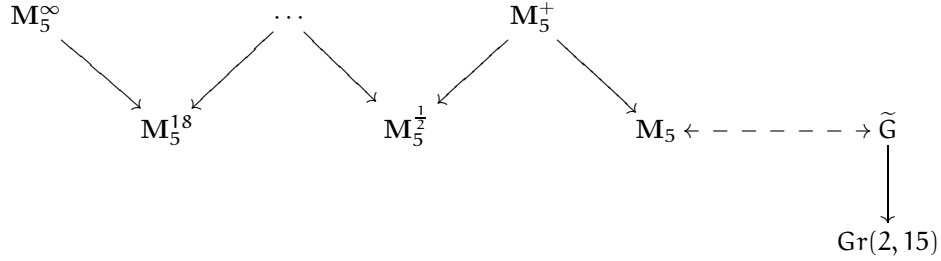
\xymatrix{\bM_5^{\infty}\ar[rd]& & \cdots\ar[ld]\ar[rd]&&\bM_5^{+}\ar[ld]\ar[rd]&&\\
&\bM_5^{18}&&\bM_5^{\frac{1}{2}}&&\bM_5\ar@{<-->}[rr]&&\widetilde{G}\ar[d]\\
&&&&&&&Gr(2,15)
}
$$
As a direct corollary,
\begin{coro}\label{maincor}
The virtual Poincar\'e polynomial of the space $\bM^{\infty}(5,2)$ is given by
\begin{multline*}
1+3p+9p^2+22p^3+50p^4+99p^5+173p^6+256p^7+330p^8+379p^9+407p^{10}\\
+420p^{11}+426p^{12}+428p^{13}+429p^{14}+428p^{15}+423p^{16}+410p^{17}+382 p^{18}\\
+333 p^{19}+259 p^{20}+176 p^{21}+101 p^{22}+51 p^{23}+22 p^{24}+9 p^{25}+3 p^{26}+p^{27}.
\end{multline*}
\end{coro}
\begin{rema}
In particular, the virtual Euler number of $\bM^{\infty}(5,2)$ is $e(\bM^{\infty}(5,2))=6030$. But the virtual Euler number of the PT-space of local $\PP^2$ (that is, the total space of the canonical line bundle $K_{\PP^2}$) is $6060$; this is obtained using the torus localization technique (\cite{choi}). The difference $30$ comes from the Euler number of the sheaves supported on $\PP^1$ which is the complement of the zero section of local $\PP^2$. This was reported to the author by J. Choi. The author would like to thank J. Choi for the comment.
\end{rema}
\subsection{Stream of the paper}
In \S2, we study the wall crossing of the moduli spaces of $\alpha$-stable pairs on $\PP^2$ by using the classification of semistable sheaves in \cite{maican, maican5}. Also, we analyze the forgetful map $\xi$ by considering the Brill-Noether locus in $\bM(5,2)$. In \S3, we  find the last birational model of $\bM(5,2)$ by studying the effective cone of the moduli space $\bM(5,2)$. As a corollary, we obtain the Poincar\'e polynomial of the space $\bM(5,2)$ which reprove the result of \cite{yuan2}. In \S4, we compute the Poincar\'e polynomial of the relative Hilbert scheme $\bB(5,7)$ by using the result of the previous sections.

\medskip
\textbf{Acknowledgement.}
We thank the anonymous referee for valuable comments and suggestions to improve the quality of the paper. The author is partially supported by Korea NRF grant 2013R1A1A2006037.

\section{Wall crossings of the spaces $\bM^{\alpha}(5,2)$}
In this section, we firstly study the wall crossing among $\bM^{\infty}(5,2)$ and $\bM^+(5,2)$. Secondly, we analyze the forgetful map $\bM^{+}(5,2)\longrightarrow \bM(5,2)$ defined in the introduction by analyzing the Brill-Noether locus. For convenience of the reader, we state the following useful results which will be used several times in this paper.
\begin{lemm}\cite[Corollary 1.6]{mhe}\label{defcoh}
Let $\Lambda=(s, F)$ and $\Lambda'=(s', F')$ be pairs on a smooth projective variety $X$. There exists a long exact sequence
\begin{align*}
0&\lr \Hom(\Lambda,\Lambda')\lr \Hom (F,F')\lr \Hom(s,H^0(F')/s')\\
&\lr \Ext^1(\Lambda,\Lambda')\lr \Ext^1(F,F')\lr \Hom(s,H^1(F'))\\
&\lr \Ext^2(\Lambda,\Lambda')\lr \Ext^2(F,F')\lr \Hom(s,H^2(F'))\lr \cdots.
\end{align*}
\end{lemm}

On the other hand, let $X$ be a quasi-projective variety. Let us denote by
$$
P(X)=\sum_i (-1)^i\text{dim}H^i(X)p^i
$$
the \emph{virtual} Poincar\'e polynomial of $X$. Let $e(X):=\sum_i (-1)^i\text{dim}H^i(X)$ be the \emph{virtual} Euler number of the variety $X$. The virtual Poincar\'e polynomial has the following \emph{motivic} properties.
\begin{prop}
\begin{enumerate}
\item $P(X)=P(X-Z)+P(Z)$ for a closed subvariety $Z$ of $X$.
\item Let $X$ and $Y$ be quasi-projective varieties. Let $\pi: X\lr Y$ be a Zariski locally trivial fibration with fiber $F$. Then $P(X)=P(Y)\cdot P(F)$.
\item Let $f:X\lr Y $ be a bijective morphism. Then $P(X)=P(Y)$.
\end{enumerate}
In (2), if the fiber is $F\cong Gr(k,n)$, the same conclusion holds even though $\pi$ is an analytic fibration (\cite[Lemma 3.1]{bak}).
\end{prop}

\subsection{Wall crossing between $\bM^{\infty}(5,2)$ and $\bM^{+}(5,2)$}
The possible types of strictly semistable pairs are given in the following table.
\begin{center}
\begin{tabular}{|l|p{8cm}|}
\hline
\multicolumn{2}{|l|}{$(d,\chi)=(5,2)$} \\
\hline
$\alpha$
&
Types of the JH-filtration of the pair $(1,(5,2))$ at $\alpha$\\
\hline
$18$&
$(1,(4,-2)\oplus (0,(1,4))$\\
\hline
$13$&
$(1,(4,-1))\oplus (0,(1,3))$\\
\hline
$8$&
$(1,(4,0))\oplus (0,(1,2))$\\
\hline
$3$&
$(1,(4,1))\oplus (0,(1,1))$\\
\hline
$3$&
$(1,(3,0))\oplus (0,(2,2))$\\
\hline
$3$&
$(1,(3,0))\oplus (0,(1,1))\oplus (0,(1,1))$\\
\hline
$\frac{1}{2}$&
$(1,(3,1))\oplus (0,(2,1))$\\
\hline
\end{tabular}
\end{center}
Here $(1, (d,\chi))$ (resp. $(0, (d,\chi))$) denotes the pair $(s,F)$ with a nonzero (resp. zero) section $s$ and the Hilbert polynomial $\chi(F(m))=dm+\chi$.
All the wall crossings except at $\alpha=3$ are simple. The wall occurs by following the configuration of points in quintic curves (Remark \ref{geo}). In this subsection, we will describe the wall crossing for the computation of the virtual Poincar\'e polynomial of the space $\bM^{\infty}(5,2)$.

Let us denote the $C_{\alpha}^{+}$ (resp. $C_{\alpha}^{-}$) by the wall crossing locus of the moduli space $\bM^{\alpha -\epsilon}(5,2)$ (resp. $\bM^{\alpha +\epsilon}(5,2)$) for sufficient small $\epsilon>0$. During the following lemmas, we use that $\bM(1,\chi)\cong \bM(1,1)$ by $F\mapsto F(-\chi+1)$ and $\bM(1,1)\cong \PP^2$. Let us start with the study of the wall crossing at $\alpha=18$. It turns out that the wall crossing locus $C_{18}^{-}$ is \emph{not} a projective bundle over its base space.
\begin{lemm}\label{wall18}
The wall crossing locus $C_{18}^{+}$ at $\alpha=18$ is a $\PP^{7}$-bundle over the product space $\PP^2\times \PP^{14}$. The locus $C_{18}^{-}$ is a $\PP^{3}$-bundle over $\PP^2\times \PP^{14}-D$ where $D= \PP^2\times \PP^{9}$ and a $\PP^4$-bundle over $D$.
\end{lemm}
\begin{proof}
By the analysis of the wall at $\alpha=18$, the $18+\epsilon$-stable pairs $(1,F)$ in $C_{18}^{+}$ fits into a non-split exact sequence
\begin{equation}\label{eq1}
\ses{(0,F_{m+4})}{(1,F)}{(1,F_{4m-2})},
\end{equation}
where $F_{dm+\chi}$ denotes any semistable sheaf with Hilbert polynomial $dm+\chi$. Also, one can easily check that all the pairs fitting in a non-split exact sequence as \eqref{eq1} are $\alpha+\epsilon$-stable. Thus the wall $C_{18}^+$ is a $\PP(\Ext^1((1,F_{4m-2}),(0,F_{m+4})))$-bundle over $\bM^{\infty}(4,-2)\times \bM(1,4)\cong\PP^{14}\times\PP^2$. Here, $\bM^{\infty}(4,-2)\cong \bB(4,0)=\PP^{14}$ by \cite[Lemma 2.3]{cc1} and $\bM(1,4)\cong \bM(1,1)=\PP^2$ by $F\mapsto F(-3)$. Let $\chi(F)=h^0(F)-h^1(F)$. Let $\chi(F,F')=\text{dim}\Ext^0(F,F') -\text{dim}\Ext^1(F,F')+\text{dim}\Ext^2(F,F')$. Since $\Ext^0((1,F_{4m-2}),(0,F_{m+4}))=0$ and $H^1(F_{m+4})=0$, from the exact sequence in Lemma \ref{defcoh}, we obtain that
$$
\text{dim}\Ext^1((1,F_{4m-2}),(0,F_{m+4}))=h^0(F_{m+4})-\chi(F_{4m-2},F_{m+4}).
$$
Note that $F_{4m-2}\cong \cO_C$ and $F_{m}\cong\cO_L(3)$ for some quartic curve $C$ and a line $L$. By using the resolution of $F_{4m-2}$, we obtain that $\text{dim}\Ext^1((1,F_{4m-2}),(0,F_{m+4}))=8$.

Similar argument shows that the wall $C_{18}^-$ is a $\PP(\Ext^1((0,F_{m+4}),(1,F_{4m-2})))$-fibration over $\bM^{\infty}(4,-2)\times \bM(1,4)\cong\PP^{14}\times\PP^2$. From the exact sequence in Lemma \ref{defcoh} again, one can see that
\begin{equation}\label{eq5}
\Ext^1((0,F_{m+4}),(1,F_{4m-2}))\cong \Ext^1(F_{m+4},F_{4m-2}).
\end{equation}
From the short exact sequence $\ses{\cO(2)}{\cO(3)}{F_{m+4}}$,
\begin{align*}
0\lr&\Ext^1(F_{m+4},F_{4m-2})\lr H^1(F_{4m-2}(-3))\lr H^1(F_{4m-2}(-2))\lr\\
& \Ext^2(F_{m+4},F_{4m-2})\lr 0.
\end{align*}
By Serre duality, $\Ext^2(F_{m+4},F_{4m-2})\cong \Ext^0(F_{4m-2},F_{m+1})$. But the later space is zero if $L \nsubseteq C$ and $\CC$ otherwise. So we have,
\[\Ext^1(F_{m+4},F_{4m-2})\simeq \begin{cases}
\CC^4 & \text{if }L \nsubseteq C,\\
\CC^5 & \text{if }L\subseteq C.\\
\end{cases} \]
Applying this fact in \eqref{eq5}, we get the result.
\end{proof}
\begin{rema}
The moduli space $\bM^{\infty}(5,2)$ is not smooth. In fact, let $(1,F)$ be a $\infty$-stable pair fitting into a non-split exact sequence in \eqref{eq1} such that $F_{m+4}\cong\cO_L(3)$ and $F_{4m-2}\cong\cO_{C\cdot L}$ for some line $L$ and cubic curve $C$.
Applying the functor $\Ext^\bullet(-,(1,F))$ (resp. $\Ext^\bullet((0, F_{m+4}),-)$) to \eqref{eq1}, we obtain
\[
\Ext^2((1,F_{4m-2}),(1,F))\lr \Ext^2((1,F),(1,F)) \stackrel{a}{\rightarrow} \Ext^2((0, F_{m+4}),(1,F))
\]
(resp.
\[
\Ext^2((0, F_{m+4}),(0,F_{m+4}))\lr \Ext^2((0, F_{m+4}),(1,F))\stackrel{b}{\rightarrow} \Ext^2((0, F_{m+4}),(1, F_{4m-2}))).
\]
By some diagram chasing and Lemma \ref{defcoh}, one can check that the composition map
\[
b\circ a:\Ext^2((1,F),(1,F))\lr \Ext^2((0,F_{m+4}),(1,F_{4m-2}))
\]
is an isomorphism. The second term $\Ext^2((0,F_{m+4}),(1,F_{4m-2}))\cong \Ext^2(F_{m+4},F_{4m-2})$ is isomorphic to $\CC$ by the proof of Lemma \ref{wall18}.
Also $\Ext^0((1,F),(1,F))\cong \CC$ by the $\infty$-stability of the pair $(1,F)$. But $\chi((s,F),(s,F)):=\sum_{i} (-1)^i\text{dim}\Ext^i((s,F),(s,F))=-26$ for all $(s,F)\in \bM^{\infty}(5,2)$ by Lemma \ref{defcoh} and \cite[Lemma 6.13]{HL}. Therefore, we have $\Ext^1((1,F),(1,F))\cong \CC^{28}$. This implies that $\bM^{\infty}(5,2)$ is not smooth at $(1,F)$ because by \cite[Lemma 4.10]{mhe} we have $$\text{dim}_{(1,F)}\bM^\infty (5,2)=27< 28=\text{dim}T_{(1,F)}\bM^\infty (5,2).$$
\end{rema}
\begin{lemm}\label{wall13}
\begin{enumerate}
\item The wall crossing locus $C_{13}^{+}$ (resp. $C_{13}^{-}$) at $\alpha=13$ is a $\PP^{6}$ (resp. $\PP^{3}$)-bundle over the product space $\PP^2\times \bB(4,1)$.
\item The locus $C_{8}^{+}$ (resp. $C_{8}^{-}$) at $\alpha=8$ is a $\PP^{5}$ (resp. $\PP^{3}$)-bundle over the product space $\PP^2\times \bB(4,2)$.
\end{enumerate}
\end{lemm}
\begin{proof}
One can easily check that, if the pair $(1,F_{4m-1})$ (resp. $(1,F_{4m}$) is semistable, so is $F_{4m-1}$ (resp. $F_{4m}$).  Hence the descriptions of the base spaces come from the fact that $\bM^{\alpha}(4,-1)\cong \bB(4,1)$ and $\bM^{\alpha}(4,0)\cong \bB(4,2)$ for all $\alpha$.
Also, the free resolutions of $F_{dm+\chi}$ are given in \cite{maican}.
$$
\ses{2\cO(-2)}{\cO(-1)\oplus \cO}{F_{4m-1}},
$$
$$
\ses{\cO(-2)\oplus \cO(-3)}{\cO\oplus \cO(-1)}{F_{4m}}.
$$
Using this fact and Lemma \ref{defcoh}, we obtain that
\begin{itemize}
\item $\Ext^1((1,F_{4m-1}),(0,F_{m+3}))\cong \CC^7$,
\item $\Ext^1((0,F_{m+3}),(1,F_{4m-1}) )\cong \CC^3$,
\item $\Ext^1((1,F_{4m}),(0,F_{m+2}))\cong \CC^6$, and
\item $\Ext^1((0,F_{m+2}),(1,F_{4m}))\cong \CC^3$.
\end{itemize}
So we have the result in the claim.
\end{proof}
Recall that the wall types at $\alpha=3$ are given by
$$
(1,(4,1))\oplus (0,(1,1)), (1,(3,0))\oplus (0,(2,2))\text{ or } (1,(3,0))\oplus (0,(1,1)) \oplus (0,(1,1)).
$$
Since the wall is not simple, we need more detail calculation.
Obviously, the first two types are general case. The third one is the intersection part. Let $A^+$ (resp. $A^-$) be the locus of the $3+\epsilon$ (resp. $3-\epsilon$)-stable pairs whose JH-filtration type is the first one.
Let $B^+$ (resp. $B^-$) be the locus of the $3+\epsilon$ (resp. $3-\epsilon$)-stable pairs whose JH-filtration type is the second one.
Let $C_3^{+}=A^{+}\cup B^{+}$ and $C_3^{-}=A^{-}\cup B^{-}$.
Let $D^+$ (resp. $D^-$) be the locus of the $3+\epsilon$ (resp. $3-\epsilon$)-stable pairs $(1,F)$ fitting into a non-split exact sequence
$$
\ses{(0,F_{2m+2})}{(1,F)}{(1,F_{3m})}
$$
$$
(\text{resp. } \ses{(1,F_{3m})} {(1,F)}{(0,F_{2m+2})})
$$
such that $F_{2m+2}=F_{m+1} \oplus F'_{m+1}$.
Since the stable pairs in the intersection part may have non-trivial automorphism, we compute the wall crossing separately in Lemma \ref{wallc31} and Lemma \ref{wallc32}.
\begin{lemm}\label{wallc31}
\begin{enumerate}
\item  \begin{enumerate}
         \item The locus $A^+$ is a $\PP^4$-bundle over $\PP^2\times B(4,3)$. The locus $A^+\cap D^+$ is a disjoint union of a $\PP^3$-bundle over a $\PP^3$-bundle over $(\PP^2\times \PP^2-\Delta)\times \bB(3,0)$ (where $\Delta$ is the diagonal of $\PP^2\times \PP^2$) and of a $\PP^2$-bundle over a $\PP^3$-bundle over $\PP^2\times \bB(3,0)$.
         \item The locus $B^+-A^+$ is a $\PP^7$-bundle over $\bM^+(3,0)\times \bM(2,2)^s$. Here the space $\bM(2,2)^s$ consists of the stable sheaves which is isomorphic to $\PP^5-V$ where the $V\cong \text{Sym}^{2}(\PP^2)$ is the space of degenerated conics.
       \end{enumerate}

\item  \begin{enumerate}
           \item The locus $A^-$ is a $\PP^3$-bundle over $\bM(1,1)\times \bM^{+}(4,1)$. The locus $A^-\cap D^-$ is a disjoint union of a $\PP^2$-bundle over a $\PP^2$-bundle over $(\PP^2\times\PP^2-\Delta)\times \bB(3,0)$ and of a $\PP^1$-bundle over a $\PP^2$-bundle over $\PP^2\times \bB(3,0)$.
         \item The locus $B^--A^-$ is a $\PP^5$-bundle over $\bM^+(3,0)\times \bM(2,2)^s$.
       \end{enumerate}
\end{enumerate}
\end{lemm}
\begin{proof}
For $\alpha=3+\epsilon$, the $\alpha$-stable pairs $(1,F)$ in $A^+$ fit into a non-split exact sequence
\begin{equation}\label{eq2}
0\lr(0,F_{m+1})\lr(1,F)\lr (1,F_{4m+1})\lr0.
\end{equation}
Also one can easily check that all of the non-split extension in the equation above are $\alpha$-stable. Thus $A^{+}$ is a $\PP(\Ext^1((1,F_{4m+1}),(0,F_{m+1})))$-bundle over $\bM(1,1)\times \bM^{\infty}(4,3)$. Note that $\bM^{\alpha}(4,1)\cong \bM^\infty(4,1)\cong \bB(4,3)$ for $\alpha >3$. By direct computation, we know that
\begin{equation}\label{eq111}
\Ext^1((1,F_{4m+1}),(0,F_{m+1}))\cong \CC^5.\end{equation}
If $(1,F)\in A^+\cap D^{+}$ in \eqref{eq2}, the pair $(1,F_{4m+1})$ should fit into the exact sequence
\begin{equation}\label{eq3}
0\lr(0,F_{m+1}')\lr (1,F_{4m+1})\lr (1,F_{3m})\lr0.
\end{equation}
By the long exact sequence obtained by \eqref{eq3}, we see
$$\Ext^1((1,F_{4m+1}),(0,F_{m+1})) \stackrel{\xi}{\rightarrow}\Ext^1((0,F_{m+1}'),(0,F_{m+1}))\lr \Ext^2((1,F_{3m}),(0,F_{m+1})).$$
But the last term is $\Ext^2((1,F_{3m}),(0,F_{m+1}))=0$ because $H^1(F_{m+1})=0$ and $\Ext^2(F_{3m},F_{m+1})\cong\Ext^0(F_{m+1},F_{3m}(-3))=0$ by the stability of $F_{m+1}$. That is, the map $\xi$ is surjective.
Then the central term $(1,F)$ of the non split extension \eqref{eq2} lie in the space $D^+$ if and only if $\xi$ is zero if applied to the class of \eqref{eq2}.
This is because, by definition, the image of the class of \eqref{eq2} in $\Ext^1((1,F_{4m+1}),(0,F_{m+1}))$ by $\xi$ corresponds to the pullback class of \eqref{eq2} in $\Ext^1((0,F_{m+1}'),(0,F_{m+1}))$ via the morphism $(0,F_{m+1}')\hookrightarrow (1,F_{4m+1})$.

But we know that
\[\Ext^1((0, F_{m+1}'),(0,F_{m+1}))= \Ext^1(F_{m+1}',F_{m+1})\simeq \begin{cases}
\CC & \text{if }F_{m+1}' \neq F_{m+1},\\
\CC^2 & \text{if } F_{m+1}'=F_{m+1}.\\
\end{cases} \]
Thus the kernel of $\xi$ depends on the choices of $F_{m+1}'$ and $F_{m+1}$.
Note that the classes of non-split extensions as \eqref{eq3} are parameterized by \begin{equation}\label{eq112}\PP(\Ext^1((1,F_{3m}),(0,F'_{m+1})))\cong \PP^3.\end{equation} Combining with this fact, we get the result (1)-(a).

The stable pairs in $B^+-A^+$ are supported on a quintic curve with smooth conic as a component. Hence,
the locus $B^{+}-A^{+}$ is a $\PP(\Ext^1((1,F_{3m}),(0,F_{2m+2}))$-bundle over $\bM^+(3,0)\times \bM(2,2)^s$. By using the resolution of the sheaves, we see that \begin{equation}\label{eq113}\Ext^1((1,F_{3m}),(0,F_{2m+2}))\cong \CC^{8}\end{equation} and so we get (1)-(b).

The proof of the case $\alpha=3-\epsilon$ is the same as that of $\alpha=3+\epsilon$ except that
\begin{itemize}
\item $\Ext^1((0,F_{m+1}),(1,F_{4m+1}))\cong \CC^4$,
\item $\Ext^1((0,F'_{m+1}),(1,F_{3m}))\cong \CC^3$, and
\item $\Ext^1((0,F_{2m+2}),(1,F_{3m}))\cong \CC^6$.
\end{itemize}
By replacing these extensions with that of \eqref{eq111}, \eqref{eq112} and \eqref{eq113}, one can finish the proof of lemma.
\end{proof}
\begin{lemm}\label{wallc32}
\begin{enumerate}
\item The locus $D^{+}$ is the disjoint union of a $\PP^3\times \PP^3$-bundle over $\PP^{9}\times (V-\overline{\Delta})$ and of a $Gr(2,4)$-bundle over $\PP^{9}\times \overline{\Delta}$. Here, $V= \text{Sym}^{2}(\PP^2)$ and $\overline{\Delta}=\PP^2$ is the diagonal of $V$.
\item The intersection locus $D^{-}$ is the disjoint the union of a $\PP^2\times \PP^2$-bundle over $\PP^{9}\times (V-\overline{\Delta})$ and of a $Gr(2,3)$-bundle over $\PP^{9}\times \overline{\Delta}$.
\end{enumerate}
\end{lemm}
\begin{proof}
The stable pairs $(s,F)\in D^+$ fit into an exact sequence
\begin{equation}\label{eq4}
\ses{(0,F_{2m+2})}{(s,F)}{(1,F_{3m})},
\end{equation}
where $F_{2m+2}=F_{m+1}\oplus F'_{m+1}$ by the definition of the locus $D^+$.
Note that a non-split extension fitting in \eqref{eq4} may not be $\alpha$-stable.
Also, the automorphism of the pair $(0,F_{2m+2})=(0,F_{m+1})\oplus (0,F_{m+1}')$ varies depending on the choice of $F_{m+1}'$ and $F_{m+1}$. Thus we handle such a situation by dividing into two cases.

If $F_{m+1}\neq F_{m+1}'$, then one can easily check that the pair $(s,F)$ is $\alpha$-stable if and only if the class of \eqref{eq4} is contained in
$$
\Ext^1((1,F_{3m}),(0,F_{2m+2}))-(\Ext^1((1,F_{3m}),(0,F_{m+1}))\cup \Ext^1((1,F_{3m}),(0,F_{m+1}'))).
$$
By quotienting out the space $\text{Aut}((0,F_{2m+2}))\cong \CC^*\times \CC^*$, we see that the space parameterizing the $\alpha$-stable pairs $(s,F)$ as above is isomorphic to the product space
$$
\PP(\Ext^1((1,F_{3m}),(0,F_{m+1})))\times \PP(\Ext^1((1,F_{3m}),(0,F_{m+1}'))).
$$
If $F_{m+1}= F_{m+1}'$, then $\Ext^1((1,F_{3m}),(0,F_{2m+2})) \cong \CC^2 \otimes \Ext^1((1,F_{3m}),(0,F_{m+1}))$ and $\text{Aut}((0,F_{2m+2}))\cong GL(2)$ acts on this $\CC^2$ in the standard way. Hence we have
$$
\Ext^1((1,F_{3m}),(0,F_{2m+2}))^s/GL(2) \cong Gr(2,\Ext^1((1,F_{3m}),(0,F_{m+1}))).
$$
Here the superscript ``s" means taking extensions corresponding to $\alpha$-stable pairs.
Since $\Ext^1((1,F_{3m}),(0,F_{m+1}))\cong \CC^4$, we have proved the second part of item (1).
The case $\alpha <3$ is the same as that of $\alpha>3$ except that $$\Ext^1((0,F_{m+1}),(1,F_{3m}))\cong \CC^3$$ so we get the results in item (2).
\end{proof}
\begin{rema}\label{rem1}
We remark that the locus satisfying the condition $F_{m+1}\neq F_{m+1}'$ in (1) (similarly in (2)) of the lemma above is not a Zariski locally trivial fibration. The wall crossing locus can be explained in a different way which enable us to compute the virtual Poincar\'e polynomial (cf. \cite{mu}). Recall that $V-\overline{\Delta}\cong (\PP^2\times \PP^2 -\Delta)/\ZZ_2$. Let $Z$ be the projective bundle over $\PP^{9}\times \PP^2$ with fiber $\PP(\Ext^1((1,F_{3m}),(0,F_{m+1})))\cong\PP^3$. The bundle $Z$ can be constructed from the tautological pair of the extensions (\cite{tomm}). Let $p:Z\times_{\PP^{9}}Z\lr \PP^{9}\times \PP^2\times \PP^2$ be the canonical projection. Then the group $\ZZ_2$ equivariantly acts on the both spaces. Let us denote the descent map by
$$
\bar{p}:Z\times_{\PP^{9}}Z/\ZZ_2\lr \PP^{9}\times (\PP^2\times \PP^2/\ZZ_2)\cong \PP^{9}\times V.
$$
Then one can easily see that the inverse image $\bar{p}^{-1}(\PP^{9}\times (V-\overline{\Delta}))$ is exactly the $\PP^{3}\times \PP^{3}$-fibration over $\PP^{9}\times (V-\overline{\Delta})$ , which is isomorphic to the quotient space $(Z\times_{\PP^{9}}Z-(p\times_{\PP^{9}} p)^{-1}(\Delta))/\ZZ_2$. Applying the formula in \cite[Lemma 2.6]{mu2}, one can get the virtual Poincar\'e polynomial of the later space.
\end{rema}
For later use, let us compute the variation of the virtual Poincar\'e polynomial at the wall $\alpha=3$.
\begin{coro}\label{threewall}
\begin{multline}\label{c3}
P(C_{3}^{+})-P(C_{3}^{-})=p^4+4p^5+13p^6+27p^7+44p^8+57p^9+66p^{10}+70p^{11}+72p^{12}\\
+72p^{13}+72p^{14}+72p^{15}+70p^{16}+66p^{17}+57p^{18}+44p^{19}+27p^{20}+13p^{21}+4p^{22}+p^{23}.
\end{multline}
\end{coro}
\begin{proof}
The wall crossing terms are a disjoint union of the locally closed subsets. Thus,
\[
\begin{split}
P(C_{3}^{+})-P(C_{3}^{-})&=[P(A^+-A^+\cap D^+) -P(A^--A^-\cap D^-)]\\
&+[P(B^+-A^+)-P(B^--A^-)]+[P(D^+)-P(D^-)].
\end{split}
\]
By the descriptions in Lemma \ref{wallc31}, Lemma \ref{wallc32} and Remark \ref{rem1}, we obtain the result.
\end{proof}
Since the proof of the lemma below is very similar to that of Lemma \ref{wall13}, we omit the proof.
\begin{lemm}\label{walllast}
The flipping locus $C_{\frac{1}{2}}^{+}$ (resp. $C_{\frac{1}{2}}^{-}$) at $\alpha=\frac{1}{2}$ is a $\PP^{6}$ (resp. $\PP^{5}$)-bundle over the product $\PP^5\times \bB(3,1)$.
\end{lemm}
In summary, through Lemma \ref{wall18}, \ref{wall13}, \ref{wallc31}, \ref{wallc32} and Lemma \ref{walllast}, the first part of Theorem \ref{mainthm} has been proved.
\begin{rema}\label{geo}
The wall crossing loci $C^+_{\alpha}$ for each $\alpha$ can be described in a geometric way (cf. \cite{cc1}). The wall crossing loci are the loci of pairs of seven points, six points, five, four points on a line with a quartic curve at the wall $\alpha=18,13,8,3$, respectively, and six points on a conic curve with a cubic curve at $\alpha=\frac{1}{2}$.
\end{rema}

\subsection{Stratification of the moduli space $\bM(5,2)$}
In this subsection, we will study the forgetful map $\bM^{+}(5,2)\longrightarrow \bM(5,2), (s,F)\mapsto F$ by using the stratification of stable sheaves in $\bM(5,2)$ (\cite{cc1.5}).
\begin{proof}[Proof of (2) in Theorem \ref{mainthm}]
By \cite[Theorem 1.1]{cc1.5}, we know that $h^0(F)\leq 3$ for all stable sheaves $F\in \bM(5,2)$. On the other hand, since $(5,2)=1$, there exists a universal family of sheaves $\cF$ on $\bM(5,2)\times \PP^2$ (\cite{lepot1}). Therefore the Proj of the direct image sheaf $p_*\cF$ is isomorphic to $\bM^{+}(5,2)$ and thus the moduli space $\bM^{+}(5,2)$ is decomposed into locally closed subsets:
\begin{enumerate}
\item the $\PP^1$-bundle over $\bM(5,2)_2$ and
\item the $\PP^2$-bundle over $\bM(5,2)_3$
\end{enumerate}
where $\bM(5,2)_k:=\{F\in \bM(5,2)| h^0(F)=k\}$.
\end{proof}
For later use, we compute the Poincar\'e polynomial of the exceptional locus $\bM(5,2)_3$.
\begin{prop}\label{poincare3}
The (virtual) Poincar\'e polynomial of the space $\bM(5,2)_3$ is given by
\begin{multline*}
1+3p+8p^2+14p^3+19p^4+21p^5+22p^6+22p^7+22p^8+22p^9+22p^{10}+22p^{11}\\
+22p^{12}+22p^{13}+22p^{14}+22p^{15}+22p^{16}+22p^{17}+21p^{18}+19p^{19}+14p^{20}+8p^{21}+3p^{22}+p^{23}.
\end{multline*}
\end{prop}
\begin{proof}
The locus $\bM(5,2)_3$ is isomorphic to the moduli space $\bM(5,-2)_1$ by \cite[Proposition 4.2.7]{choi}. Also, one can easily check that the forgetful map $\xi:\bM^+(5,-2)\lr \bM(5,-2)$ is injective and onto the space $\bM(5,-2)_1$ by using \cite[Table 1]{maican5}. Moreover, the map $\xi$ is a closed embedding since the differential map $\xi_*:\Ext^1((s,F),(s,F))\lr\Ext^1(F,F)$ is injective by $\Hom(s,H^0(F)/(s))=0$ (Lemma \ref{defcoh}). Therefore, $$\bM^+(5,-2)\cong \bM(5,-2)_1\cong \bM(5,2)_3.$$ Let us compute the polynomial of the moduli space $\bM^{+}(5,-2)$ by using the wall crossings.
Among the moduli spaces $\bM^{\alpha}(5,-2)$, one can easily see that there is a single wall crossing at $\alpha=2$ such that the JH-filtration is given by $(1,(5,-2))=(1,(4,-2))\oplus (0,(1,0))$. Also, by \cite[Lemma 2.3]{cc1}, the space $\bM^{\infty}(5,-2)$ is a projective bundle over $Hilb^{3}(\PP^2)$ with fiber $\PP^{17}$. Since $\Ext^1((1,F_{4m-2}),(0,F_{m}))=\Ext^1((0,F_{m}),(1,F_{4m-2}))=\CC^4$, we get
$$
P(\bM(5,2)_3)=P(\bM^{+}(5,-2))=P(\bM^{\infty}(5,-2))+(P(\PP^3)-P(\PP^{3}))\cdot P(\PP^2)\times P(\PP^{14}).
$$
Also $P(Hilb^3(\PP^2))=1+2q+5q^2+6q^3+5q^4+2q^5+q^6$ (\cite{ell}), so the claim is proved.
\end{proof}
\section{Bridgeland wall crossing of the moduli space $\bM(5,2)$}
In this section, we study the wall crossing of the space $\bM(5,2)$ in the sense of Bridgeland. For the detail of the Bridgeland wall crossing, see \cite{woolf}.
The wall crossing of $\bM(5,2)$ can be done similarly to \cite{cc2}. So we omit the detail about the wall computations. From now on, we focus on finding the final birational model of $\bM(5,2)$.
To solve this, let us describe the ray generator of the effective cone of the moduli space $\bM(5,2)$. As a set, the divisor $D$ is defined as the locus of stable sheaves which is \emph{not} orthogonal to the vector bundle $E$ (for detail, see \cite{woolf}). The existence of such a vector bundle $E$ has been proved in \cite[Theorem 4.3]{woolf}. Let $A:=\phi^*\cO(1)$ where the map $\phi:\bM(5,2)\longrightarrow |\cO_{\PP^2}(5)|$ is defined by the Fitting ideal (\cite{lepot1}). Obviously, the divisor $A$ is the nef divisor of the moduli space $\bM(5,2)$.
\begin{lemm}
The effective cone of $\bM(5,2)$ is generated by the two geometric divisors $A$ and $D=\overline{X_{01}}$. Here the locus $X_{01}$ consists of the stable sheaves of the forms $\cO_C(2)(-Z_4+Z_1)$ such that $C$ is a smooth quintic curve and $Z_i$ is the subscheme of $C$ with length $i$ in a general position.
\end{lemm}
\begin{proof}
By \cite[Theorem 5.3]{woolf}, the divisors $A$ and $D$ generate the rays of the effective cone of the space $\bM(5,2)$.
On the other hand, the general free resolution of the stable sheaf $F\in \bM(5,2)$ has two types depending on some algebraic conditions (\cite[\S 2.3]{maican5}). One can easily check that the sheaf $F$ is orthogonal to $E$ if and only if $F$ fits into the exact sequence $\ses{T_{\PP^2}(-4)}{2\cO_{\PP^2}}{F}$. Hence the complement $X_{01}$ consisting of the stable sheaves $\cO_C(2)(-Z_4+Z_1)$ (\cite[Proposition 2.3]{maican5}) is exactly the support of the divisor $D$.
\end{proof}
\begin{prop}\label{bridge}
The final birational model of the moduli space $\bM(5,2)$ is isomorphic to the Grassmannian variety $Gr(2,15)$.
\end{prop}
\begin{proof}
From \cite[\S 9.2]{dretrau} and \cite[\S 2.3]{maican5}, the blown-up space  $\widetilde{G}$ of $Gr(2,15)$ along a $\PP^2\times Gr(2,6)$ is isomorphic to $\bM(5,2)$ up to codimension one because the exceptional divisor $E$ is supported on the strict transformation of $X_{01}$ in $\bM(5,2)$. Hence $\text{Eff}(\bM(5,2))= \text{Eff}(\widetilde{G})$. Let us compute the corresponding divisor at the wall $W$ which is right before the collapsing one. Let us denote by $\lambda$ the map $K(\PP^2)\lr \text{Pic}(\bM(5,2))$ which is defined by the Fourier-Mukai transformation (for detail, see \cite{cc2, woolf}). Then $A=\lambda(h^2)$ and $D=\lambda(-5+2h+3h^2)$ for the class $h=[\cO_l]$ of a line $l\subset \PP^2$ (\cite{woolf}). Also, the destabilizing objects at the wall $W$ are of type $[\cO(-2)\lr 2O]$ (\cite[Table 1]{maican5}). The same computation as did in \cite[Remark 2.12]{cc2} tells us that $A+D$ is the corresponding divisor at the wall $W$.

On the other hand, on $\widetilde{G}$, $$-15A=K_{\bM(5,2)}=\pi^*K_{Gr(2,15)}+15E.$$ The first equality comes from \cite[Lemma 3.1]{woolf} and the second one comes from \cite[Excercise 8.5, II]{Hartshorne}. Hence $\pi^{*}(-K_{Gr(2,15)})=15A+15D$ is a nef (but not ample) divisor on $\widetilde{G}$ because $K_{Gr(2,15)}$ is anti-ample and $D=E=\overline{X_{01}}$. Thus the corresponding birational model of the divisors in $[A+D, D)$ is the space $Gr(2,15)$.
\end{proof}
\begin{prop}
The moduli space $\bM(5,2)$ can be obtained from the space $Gr(2,15)$ by the Bridgeland wall crossings.
\end{prop}
\begin{proof}
Proposition \ref{bridge} and \cite[Theorem 1.1]{bertram} imply the statement.
\end{proof}
\begin{coro}\label{coro2}
The Poincar\'e polynomial of the space $\bM(5,2)$ is given by
\begin{align*}
&1 + 2p+6p^{2}+13p^{3}+26p^{4}+45p^{5}+68p^{6}+87p^{7}+100p^{8}+107p^{9}\\
&+111p^{10}+112p^{11}+113p^{12}+113p^{13}+113p^{14}+112p^{15}+111p^{16}+107p^{17}\\
&+100p^{18}+87p^{19}+68p^{20}+45p^{21}+26p^{22}+13p^{23}+6p^{24}+2p^{25}+p^{26}.
\end{align*}
\end{coro}
\begin{proof}
The computation of the wall crossings is similar to that of $\bM(6,1)$ in \cite{cc2}. So we omit the detail.
\end{proof}
\begin{rema}
This result is compatible with \cite[Theorem 6.1]{yuan2}.
\end{rema}
\section{Computation of the virtual Poincar\'e polynomial of $\bM^{\infty}(5,2)$}
Summing up the results of the previous sections, we obtain the virtual Poincar\'e polynomial of $\bM^{\infty}(5,2)$.
\begin{proof}[Proof of Corollary \ref{maincor}]
From part (2) in Theorem \ref{mainthm}, Corollary \ref{coro2} and Proposition \ref{poincare3}, we obtain
\begin{align*}
&P(\bM^{+}(5,2))=(P(\bM(5,2))-P(\bM(5,2)_3))\cdot P(\PP^1)+P(\bM(5,2)_3)\cdot P(\PP^2)\\
&=1+3p+9p^2+22 p^3+47 p^4+85 p^5+132 p^6+176 p^7+209 p^8+229 p^9\\
&+240 p^{10}+245 p^{11}+247 p^{12}+248 p^{13}+248 p^{14}+247 p^{15}+245 p^{16}+240 p^{17}\\
&+229 p^{18}+209 p^{19}+176 p^{20}+132 p^{21}+85 p^{22}+47 p^{23}+22 p^{24}+9 p^{25}+3 p^{26}+p^{27}.
\end{align*}
Let us add the wall crossing terms in Lemma \ref{wall18}, Lemma \ref{wall13}, Corollary \ref{threewall} and Lemma \ref{walllast}. Let $P(C_{\alpha}):=P(C_{\alpha}^+)-P(C_{\alpha}^-)$. Then,
\begin{align*}
&P(\bM^{\infty}(5,2))= P(\bM^{+}(5,2))+P(C_{18})+P(C_{13})+P(C_{8})+P(C_{3})+P(C_{\frac{1}{2}}) \\
&=P(\bM^{+}(5,2))+[P(\PP^7)P(\PP^2)P(\PP^{14})-P(\PP^3)P(\PP^2\times(\PP^{14}-\PP^9))-P(\PP^4)P(\PP^2)P(\PP^9)]\\
&+(P(\PP^6)-P(\PP^3))P(\PP^2)P(\bB(4,1))+(P(\PP^5)-P(\PP^3))P(\PP^2)P(\bB(4,2))+\eqref{c3}\\
&+(P(\PP^6)-P(\PP^5))P(\PP^5)P(\bB(3,1)).
\end{align*}
Here, $P(\bB(d,1))=P(\PP^{\frac{d(d+3)-2}{2}})\cdot P(\PP^2)$ for $d=3,4$ and $P(\bB(4,2))=P(\PP^{12})\cdot (1+2p+3p^2+2p^3+p^4)$ (\cite[Lemma 2.3]{cc1} and \cite{ell}).
The claim is proved with the help of the computer program Maple.
\end{proof}
\bibliographystyle{amsplain}

\end{document}